\documentclass[11pt]{amsproc}

\def\N{{{\Bbb N}}}
\def\Z{{{\Bbb Z}}}
\def\T{{{\Bbb T}}}
\def\R{{\Bbb R}}
\def\C{{\Bbb C}}
\def\l{{\lambda }}
\def\a{{\alpha }}
\def\D{{\Delta }}

\def\a{{\alpha}}
\def\b{{\beta}}
\def\d{{\delta}}
\def\e{{\varepsilon}}

\def\vp{{\varphi}}
\def\t{{\theta }}
\def\g{{\gamma }}
\def\w{{\omega }}
\def\L{{\Lambda }}
\def\i{{\mathrm{i} }}
\def\L{\mathcal{L}}
\def\){\right)}
\def\({\left(}

\numberwithin{equation}{section}
\newtheorem{corollary}{Corollary}[section]
\newtheorem{lemma}{Lemma}[section]
\newtheorem{theorem}{Theorem}[section]
\newtheorem{proposition}{Proposition}[section]
\newtheorem{remark}{Remark}[section]
\newtheorem{definition}{Definition}[section]

\newtheorem*{thma}{Theorem A}
\newtheorem*{thmb}{Theorem B}

\par

\begin{document}

\title{Sharp estimates of approximation of periodic functions in H\"older spaces}

\author[Yurii
Kolomoitsev]{Yurii
Kolomoitsev$^{\text{a},*,1,2}$}
\address{Institute of Mathematics of NAS of Ukraine,
Tereschenkivska st. 3, 01601 Ukraine, Kiev-4}
\email{kolomus1@mail.ru}

\author[J\"urgen~Prestin]{J\"urgen~Prestin$^\text{a,1}$}
\address{Universit\"at zu L\"ubeck,
Institut f\"ur Mathematik,
Ratzeburger Allee 160,
23562 L\"ubeck}
\email{prestin@math.uni-luebeck.de}

\thanks{$^\text{a}$Universit\"at zu L\"ubeck,
Institut f\"ur Mathematik,
Ratzeburger Allee 160,
23562 L\"ubeck}

\thanks{$^1$Supported by FP7-People-2011-IRSES Project number 295164  (EUMLS: EU-Ukrainian Mathematicians for Life Sciences).}
\thanks{$^2$Supported by the Grant of the National Academy of Sciences of Ukraine
for young scientists.}

\thanks{$^*$Corresponding author}

\thanks{E-mail address: kolomus1@mail.ru}

\date{\today}

\begin{abstract}
The main purpose of the paper is to study sharp estimates of approximation of periodic functions in the H\"older spaces $H_p^{r,\a}$ for all $0<p\le\infty$ and $0<\a\le r$. By using  modifications of the classical moduli of smoothness, we give improvements of the direct and inverse theorems of approximation and prove the criteria for the precise order of decrease of the best approximation in these spaces. Moreover, we obtained strong converse inequalities for general methods of approximation of periodic functions in $H_p^{r,\a}$.
\end{abstract}

\maketitle

\bigskip
\bigskip
\section{Introduction and notations}

Let $\T\cong[0,2\pi)$ be the torus. As usual, the space
$L_p=L_p(\T)$, $0<p<\infty$, consists of measurable complex
functions, which are $2\pi$-periodic  and
$$
\Vert f\Vert_p=\bigg(\frac1{2\pi}\int_{\T}|f(x)|^p \mathrm{d}x\bigg)^{\frac 1p}<\infty.
$$
For simplicity, by $L_\infty=L_\infty(\T)$ we denote the space of all
$2\pi$-periodic continuous functions on $\T$ which is equipped
with the norm
$$
\Vert f\Vert_\infty=\max_{x\in\T}|f(x)|.
$$

Let us consider the linear Fourier means
\begin{equation}\label{eqI1}
  \mathcal{L}_n(f,x)=\int_{\T} f(t)K_n(x-t)\mathrm{d}t,\quad x\in\T,\quad f\in L_p,\quad 1\le p\le\infty,
\end{equation}
which are generated by some kernel
\begin{equation}\label{eqI2}
  K_n(x)=\sum_{\nu=-n}^n a_{\nu,n}e^{\i \nu  x},\quad a_{0,n}=1,\quad n\in\N.
\end{equation}
To estimate the error of approximation of functions by Fourier means (\ref{eqI1}) one usually uses the classical moduli of smoothness:
$$
\w_k(f,t)_p=\sup_{0<\d<t}\Vert \D_\d^k f\Vert_p,
$$
where
$$
\D_\d^k f(x)=\sum_{\nu=0}^k (-1)^\nu \binom{k}{\nu}f(x+\nu \d).
$$

For a long time, there has been some interest in investigating the approximation of functions by the Fourier means in the H\"older spaces.
This interest originated from the study of a certain class of integro-differential equations and from applications in error estimations for singular integral equations (see, for example,~\cite{Ka} and~\cite{Pr}).

We will say that $f\in H_p^{r,\a}$, $0<p\le\infty$, $0<\a\le r$, $r\in\N$, if $f\in L_p$ and
\begin{equation}\label{eqI4}
  \Vert f\Vert_{H_p^{r,\a}}=\Vert f\Vert_p+|f|_{H_p^{r,\a}}<\infty,
\end{equation}
where
\begin{equation*}
  |f|_{H_p^{r,\a}}=\sup_{h>0}\frac{\Vert \D_h^r f\Vert_p}{h^\a}=\sup_{h>0}\frac{\w_r(f,h)_p}{h^\a}.
\end{equation*}
Following the initial works of Kalandiya~\cite{Ka} and Pr\"ossdorf~\cite{Pr} the problems of approximation in H\"older spaces $H_p^{r,\a}$ were studied by Ioakimidis~\cite{Io},  Rempulska and Walczak~\cite{RW}, Bustamante and Roldan~\cite{BR} and many others.
One can find an interesting survey on this subject in~\cite{BJ}, see also~\cite{PrPr}.

The standard results on approximation in $H_p^{r,\a}$, $1\le p\le\infty$, have the following form.
If the means $\{\mathcal{L}_n\}$ possess some good properties and $f\in H_p^{r,\a}$, then
\begin{equation}\label{eqI5}
  \Vert f-\mathcal{L}_n(f)\Vert_{H_p^{r,\a}}\le {C}\sup_{0<h\le 1/n}\frac{\w_r(f,h)_p}{h^\a},\quad n\in\N,
\end{equation}
(see, e.g., \cite{Pr}, \cite{PrPr}, \cite{RW} or \cite{BR}).
Note that, the estimates of type (\ref{eqI5}) in general are not sharp.
One can verify this by using a function $f\in L_p$ with
$f^{(s)}\in H_p^{r,\a}$ and sufficiently large~$s$.

It turns out that one can obtain more general and sharper results by using
the simple fact that any Fourier multiplier in $L_p$, $1\le p\le \infty$, is also a Fourier multiplier in the H\"older spaces $H_p^{r,\a}$.
Indeed, let $\L$ be any Fourier multiplier in $L_p$ i.e., a bounded linear operator in $L_p$, which commutates with translates.
Then
\begin{equation}\label{eqI9}
  \Vert \D_h^r (\L f)\Vert_p=\Vert \L (\D_h^r  f)\Vert_p\le \Vert \mathcal{L}\Vert_{L_p\to L_p} \Vert \D_h^r  f\Vert_p.
\end{equation}
Thus, some known  results about approximation of functions in $L_p$ spaces with $1\le p\le\infty$ can be easily transferred to the H\"older spaces. For example, let means (\ref{eqI1}) be such that for all $f\in L_p$ and $n\in\N$:
\begin{equation}\label{eqI9_1}
\Vert f-\mathcal{L}_n(f)\Vert_p\asymp \w_k\(f,\frac1n\)_p,
\end{equation}
where, as usual, $A(f, n)\asymp B(f, n)$ means that there exist positive
constants $c$ and $C$ such that $c A(f,n)\le B(f,n)\le C A(f,n)$ for all $f$ and $n$.
Then, by using (\ref{eqI9}) and (\ref{eqI9_1}), one can easily derive
\begin{equation}\label{eqI10}
  \Vert f-\mathcal{L}_n(f)\Vert_{H_p^{r,\a}}\asymp \w_k\(f,\frac1n\)_p+\sup_{h>0}h^{-\a}\w_k\(\D_h^r f,\frac1n\)_p
\end{equation}
(see Theorem~\ref{th2} below).

Note that the inequalities of the form (\ref{eqI9_1}) were first obtained by Trigub in~\cite{T68}.
See also the results of Ditzian and Ivanov~\cite{DI}, in which such inequalities are called strong converse inequalities.

It is easy to see that (\ref{eqI10}) for arbitrary $k$ independently of $r$ and with a two-sided estimate is an essential improvement of an estimate like (\ref{eqI5}). In this sense the problem of approximation in the H\"older spaces $H_p^{r,\a}$ for $1\le p\le \infty$ is  an easy problem.
But the analogous problem of approximation of functions in $H_p^{r,\a}$ with $0<p<1$ is much more difficult
because the spaces $L_p$ with $0<p<1$ are essentially different from the spaces  $L_p$ with $p\ge 1$:
as it was mentioned by Peetre~\cite{peetre} these spaces are "pathological" spaces.


%
%

The main purpose of the paper is to study sharp estimates of approximation of periodic functions in the H\"older spaces $H_p^{r,\a}$ for all $0<p\le\infty$ and $0<\a\le r$. By using  modifications of the classical moduli of smoothness, we give improvements of known direct and inverse theorems of approximation (see~\cite{BR} and~\cite{BJ99}) and prove the criteria of the precise order of decrease of the best approximation in $H_p^{r,\a}$.
We also obtain estimates like (\ref{eqI10}) for the general Fourier means (\ref{eqI1}) and for the  families of linear polynomial means
\begin{equation}\label{eqI14}
  \mathcal{L}_{n,\l}(f,x)=\frac1{4n+1}\sum_{j=0}^{4n}f(t_j+\l)K_n(x-t_j-\l),
\end{equation}
where $t_j={2\pi j}/{(4n+1)}$, $\l\in\R$, and the kernel $K_n$ is as in (\ref{eqI2}).
These means were intensively studied in several papers of Runovskii and Schmeisser (see, e.g.~\cite{RS2}, \cite{RSFaaF}).


The paper is organized as follows: The auxiliary results are formulated and proved in
Section~2. In Section~3 we prove the analogs in $H_p^{r,\a}$ of some classical theorems of approximation theory.
We also obtain some new properties of the best approximation in $H_p^{r,\a}$.
In Section~4 we prove the two-sided inequalities for approximation of functions by some linear summation methods of Fourier series and by families of linear polynomial means. In Section~5 we consider some corollaries and make concluding remarks.

We denote by $c$ and $C$  positive constants depending on the indicated parameters. We also denote $p_1=\min(p,1)$.

\section{Preliminary remarks and auxiliary results}

Let us recall several properties of the H\"older spaces $H_p^{r,\a}$, $0<p\le \infty$.

\begin{lemma}\label{lemHolder} (See~\cite[Ch. 2, Theorem 10.1 for the case $q=\infty$]{DL}.)
Let $0<p\le\infty$, $0<\a<r<k$, and $k,r\in\N$. Then the (quasi)-norms of a function in $H_p^{r,\a}$  and $H_p^{k,\a}$ are equivalent.

\end{lemma}

Lemma~\ref{lemHolder} does not hold in the case $\a=r$.
However, there exist the following descriptions of $H_p^{r,r}$ for $1\le p< \infty$.

\begin{lemma}\label{lemHolderD} (See~\cite[Ch. 2, \S 9]{DL} and~\cite{Br}.)
Let $r\in \N$. Then
\begin{enumerate}
  \item Suppose $1<p< \infty$. Then
  $f\in H_p^{r,r}$ iff $f^{(r-1)}\in AC$ and $f^{(r)}\in L_p$.
  Moreover,
  $$
  \Vert f\Vert_{H_p^{r,r}}=\Vert f\Vert_p+\Vert f^{(r)}\Vert_p.
  $$
  \item
  $f\in H_1^{r,r}$ iff $f^{(r-2)}\in AC$ (in the case $r\ge 2$) and $f^{(r-1)}$ is a function of bounded variation on $[0,2\pi]$.
  Moreover,
  $$
  \Vert f\Vert_{H_1^{r,r}}=\Vert f\Vert_1+{\rm var}_{[0,2\pi)} f^{(r-1)}.
  $$
\end{enumerate}
\end{lemma}

At the same time, if $1\le p\le \infty$, $r\in\N$, and $f\in L_p$ is such that
$
\w_r(f,h)_p=o(h^r),
$
then $f\equiv {\rm const}$. In the case $0<p<1$ the situation is totally different. In~\cite{Ra}, it was shown that for any $r\in \N$
and $\a\in (0,r-1+1/p]$ there exists a function $f\in L_p$ such that
$\w_r(f,h)_p=\mathcal{O}(h^\a)$.
In particular, for
\begin{equation}\label{eqCondModR}
f(x)=\sum_{\nu=0}^\infty\frac{\sin((2\nu+1)x-\pi(r-1)/2)}{(2\nu+1)^r}
\end{equation}
it is known from~\cite{Ra} that
\begin{equation}\label{eqCondMod}
  \w_r(f,h)_p=\mathcal{O}(h^{r-1+\frac 1p}).
\end{equation}
For the complete description of functions satisfying condition (\ref{eqCondMod}) see~\cite{Kr} for the case $r=1$ and~\cite{K03} for $r\ge 2$.

Note that, in spite of the fact that the H\"older spaces $H_p^{r,\a}$ make sense for all $\a\in (0,r-1+1/p]$ we will consider only the case $\a\le r$
(see Remark~\ref{rem1}).

Let $\mathcal{T}_n$ be the set of all trigonometric
polynomials of order at most $n$
\begin{equation*}
\mathcal{T}_n=\bigg\{T(x)=\sum_{\nu=-n}^n c_\nu {\rm e}^{\i\nu x}~:~
c_\nu \in \C\bigg\},
\end{equation*}
and let
$$
E_n(f)_p=\inf_{T\in \mathcal{T}_n}\Vert f-T\Vert_p
$$
be the error of the best approximation of a function $f$ in $L_p$ by trigonometric polynomials of order at most $n$.
A trigonometric polynomial $T_n\in \mathcal{T}_n$  is called polynomial of best approximation in $L_p$, if
$$
\Vert f-T_n\Vert_p=E_n(f)_p.
$$
Note that for any $f\in L_p$ and $n\in \N\cup\{0\}$ such polynomials always exist (see~\cite[Ch. 3, \S 1]{DL}).

Recall the Jackson-type theorem in $L_p$-spaces, see~\cite{St51} for $1\le p\le\infty$ and~\cite{SO} for $0<p<1$ (see also~\cite{SKO} and~\cite{I}).
\begin{thma}\label{lem01}
Let $f\in L_p$, $0< p\le\infty$, $k\in \N$, and $n\in\N$. Then
\begin{equation}\label{eqI12}
  E_n(f)_p\le C\w_k\(f,\frac 1n\)_p,
\end{equation}
where $C$ is a constant independent of $n$ and $f$.
\end{thma}

We will  use the following Stechkin-Nikolskii type inequality (see~\cite{DHI}).

\begin{thmb}\label{lem1}
Let $0<p\le\infty$, $n\in\N$, $0<h\le \pi/n$, and $r\in\N$. Then the following two-sided inequality holds for any trigonometric polynomial $T_n \in \mathcal{T}_n$
\begin{equation}\label{eqA1}
  h^r\Vert T_n^{(r)}\Vert_p\asymp \Vert \D_h^r T_n\Vert_p,
\end{equation}
where $\asymp$ is a two-sided inequality with absolute constants independent of $T_n$ and $h$.
Moreover, if $T_n$ is a polynomial of the best approximation of a function $f\in L_p$, then we have
\begin{equation}\label{eqA2}
\Vert \D_h^r T_n\Vert_p\le C\w_r\(f,\frac1n\),
\end{equation}
where $C$ is a constant independent of $T_n$, $h$, and $f$.
\end{thmb}

In our paper we will  often use the following two properties of moduli of smoothness
 (see \cite[Ch. 2, \S~7 and  Ch. 12, \S~5]{DL} or
\cite[Ch. 4]{TB}).
Let $f\in L_p$, $0<p\le \infty$, $r\le k$, $k,r\in\N$. Then
\begin{equation}\label{svmod1}
       \omega_k(f,h)_p\le 2^{\frac{k-r}{p_1}}\omega_r(f,h)_p\le 2^\frac k{p_1}\Vert f\Vert_p,\quad h>0,
\end{equation}
\begin{equation}\label{svmod2}
       \omega_r(f,\lambda h)_p\le r^{\frac 1{p_1}-1}(1+\lambda)^{\frac
       1{p_1}+r-1}\omega_r(f,h)_p,\quad h>0,\quad \l>0.
\end{equation}

In this paper, we will deal with functions in $L_p(\T^{2})$ which depend
additionally on a parameter $\lambda\in\T$.
We denote by $\Vert\cdot\Vert_{\overline{p}}$  the $p$-(quasi-)norm
with respect to both the main variable $x\in\T$ and the parameter
$\lambda\in\T$, i.e.
$$
\Vert\cdot\Vert_{\overline{p}}=\Vert
\Vert\cdot\Vert_{p;\,x}\Vert_{p;\,\l}\,,
$$
where $\Vert\cdot\Vert_{p;\,x}$ and $\Vert\cdot\Vert_{p;\,\l}$ are
the $p$-norms (quasi-norms, if $0<p<1$) with respect to $x$ and
$\l$, respectively.

We will understand the error of approximation of $f\in L_p$ by a family of linear operators
$\{\mathcal{L}_{n,\,\lambda}\}_{n\in\N,\, \lambda\in\T}$ in the H\"older spaces $H_p^{r,\a}$ in the following sense:
\begin{equation}\label{eqMeansH}
 \Vert f-\mathcal{L}_{n,\,\lambda}(f)\Vert_{H_{\overline{p}}^{r,\a}}=\Vert
f-\mathcal{L}_{n,\,\lambda}(f)\Vert_{\overline{p}}+|f-\mathcal{L}_{n,\,\lambda}(f)|_{H_{\overline{p}}^{r,\a}},
\end{equation}
where
\begin{equation}\label{eqMeansH2}
  |f-\mathcal{L}_{n,\,\lambda}(f)|_{H_{\overline{p}}^{r,\a}}=\sup_{h>0}\frac{\Vert
\D_h^r (f-\mathcal{L}_{n,\,\lambda}(f))\Vert_{\overline{p}}}{h^\a}.
\end{equation}
As usual, the norm of a linear and bounded operator $\mathcal{L}_n$ in $L_p$ is given by
\begin{equation*}
  \Vert \L_n\Vert_{(p)}=\sup_{\Vert f\Vert_p\le 1}\Vert \mathcal{L}_n(f)\Vert_p.
\end{equation*}
A sequence $\{\mathcal{L}_n\}_{n\in\N}$ is said to be bounded if the sequence of its norms is bounded by some positive constant independent of $n$.
Now let us consider the family of operators $\{\mathcal{L}_{n,\,\lambda}\}$ in $L_p(\T^2)$. We define the (averaged) norm of such a family by
\begin{equation*}
 \Vert \{\mathcal{L}_{n,\l}\}\Vert_{(p)} =\sup_{\Vert f\Vert_p\le 1}\Vert \mathcal{L}_{n,\l}(f,x)\Vert_{\overline{p}}.
\end{equation*}
By the analogy, a family $\{\mathcal{L}_{n,\l}\}$ is said to be bounded if the sequence of its norms is bounded. Everywhere below $\{\mathcal{L}_{n,\,\lambda}\}$ stands  for  families of linear polynomial means of type (\ref{eqI14}).

Let us consider the inequalities of the type (\ref{eqI9}).
It is well-known that there are no non-trivial Fourier multipliers
in $L_p$ with $0<p<1$, but even if we replace in (\ref{eqI9}) the operator $\L$ by some family  $\{\L_{n,\l}\}$, this inequality
does not hold in $L_p$, $0<p<1$. Indeed, suppose the following inequality holds for some non-trivial family  $\{\L_{n,\l}\}$
\begin{equation*}
\(\int_\T {\Vert \D_h^r \mathcal{L}_{n,\l}(f)\Vert_p^p}\mathrm{d}\l\)^\frac1p \le C_{r,p} \Vert \D_h^r f\Vert_p, \quad f\in L_p,\quad 0<p<1.
\end{equation*}
Note that, for the function $f$ from (\ref{eqCondModR})
we have $\Vert \D_h^r f\Vert_p=\mathcal{O}(h^{r-1+1/p})$, but by Theorem~B one only has that
$\Vert \D_h^r \L_{n,\l}(f)\Vert_{\overline{p}}\asymp h^r$ as $h\to 0$. Thus, we have a contradiction.

However, the following lemma holds.
\begin{lemma}\label{lem2}
Let $f\in L_p$, $0<p\le \infty$, $0< \a \le r$, $r,n\in \N$, and let the family $\{\mathcal{L}_{n,\l}\}$ be bounded in $L_p$. Then
\begin{equation}\label{eqA3}
 \(\int_\T \sup_{h>0}\Vert {h^{-\a} \D_h^r \mathcal{L}_{n,\l}(f)\Vert_p^p}\mathrm{d}\l\)^\frac1p\le C\sup_{h\ge 1/n}\frac{\omega_r(f, h)_p}{h^{\a}},
\end{equation}
with the usual modification in the case $p=\infty$, where $C$ is a constant independent of $\a$, $f$ and $n$.
\end{lemma}

\begin{proof}
We treat only the case $0<p<\infty$; similar arguments apply when $p=\infty$.

Let us first prove that for any $n\in \N$
\begin{equation}\label{eqlem2pr1}
 \(\int_\T \sup_{h\ge 1/n}\Vert {h^{-\a} \D_h^r \mathcal{L}_{n,\l}(f)\Vert_p^p}\mathrm{d}\l\)^\frac1p \le C\sup_{h\ge 1/n}\frac{\omega_r(f, h)_p}{h^{\a}}.
\end{equation}
By~\cite[p. 28]{RSFaaF}, we have
\begin{equation}\label{eqlem2pr300}
  \begin{split}
\mathcal{L}_{n,\l}(T_n)=\mathcal{L}_{n} (T_n)
  \end{split}
\end{equation}
for any $T_n\in \mathcal{T}_n$, $n\in \N$, and $\l\in \R$.
Therefore, we have also
\begin{equation}\label{eqlem2pr3}
  \begin{split}
\D_h^r \mathcal{L}_{n,\l}(f)&=\D_h^r \mathcal{L}_{n,\l}(f-T_n)+\D_h^r \mathcal{L}_{n,\l}(T_n)=\D_h^r \mathcal{L}_{n,\l}(f-T_n)+\mathcal{L}_{n,\l}(\D_h^r T_n)
  \end{split}
\end{equation}
for any $h>0$ and $\l\in \R$.
By Theorem~A, we  choose polynomials $T_n \in \mathcal{T}_n$, $n\in\N$, such that
\begin{equation}\label{eqlem2pr2}
  \Vert f-T_n\Vert_p\le C \w_r\(f,\frac1n\)_p.
\end{equation}
Thus, using (\ref{eqlem2pr300})--(\ref{eqlem2pr2}), and the boundedness of $\{\mathcal{L}_{n,\l}\}$, we obtain (\ref{eqlem2pr1}) by
\begin{equation*}\label{eqlem2pr4}
  \begin{split}
    \int_\T \sup_{h\ge 1/n}&\frac{\Vert  \D_h^r \mathcal{L}_{n,\l}(f)\Vert_p^p}{h^{\a p}}\mathrm{d}\l\\
   &\le C\left(\int_\T \sup_{h\ge 1/n}\frac{\Vert \D_h^r \mathcal{L}_{n,\l}(f-T_n)\Vert_p^p}{h^{\a p}} \mathrm{d}\l + \int_\T \sup_{h\ge 1/n}\frac{\Vert  \mathcal{L}_{n}(\D_h^r T_n)\Vert_p^p}{h^{\a p}} \mathrm{d}\l\right)\\
      &\le C\left(n^{\a p} \int_\T \Vert  \mathcal{L}_{n,\l}(f-T_n)\Vert_p^p \mathrm{d}\l + \sup_{h\ge 1/n}\int_\T \frac{\Vert  \mathcal{L}_{n,\l}(\D_h^r T_n)\Vert_p^p}{h^{\a p}} \mathrm{d}\l\right)\\
&\le C\Vert\{{\L}_{n,\l}\} \Vert_{(p)}^p\( n^{\a p}\Vert f-T_n\Vert_p^p+\sup_{h\ge 1/n} \frac{\Vert \D_h^r T_n\Vert_p^p}{h^{\a p}} \)\\
&\le C\(n^{\a p} \Vert f-T_n\Vert_p^p+\sup_{h\ge 1/n} \frac{\Vert \D_h^r f\Vert_p^p}{h^{\a p}}\)\\
&\le C\(n^{\a p}\w_r\(f,1/n\)_p^p+\sup_{h\ge 1/n} \frac{\Vert \D_h^r f\Vert_p^p}{h^{\a p}} \)
\le C \sup_{h\ge 1/n} \frac{\w_r\(f,h\)_p^p}{h^{\a p}}.
  \end{split}
\end{equation*}

Now, by using Theorem~B and (\ref{eqlem2pr1}), we derive
\begin{equation*}
  \begin{split}
\int_\T \sup_{h>0}\frac{\Vert \D_h^r \mathcal{L}_{n,\l}(f)\Vert_p^p}{h^{\a p}}\mathrm{d}\l
&\le  \int_\T \left\{\sup_{0<h< 1/n}+\sup_{h\ge 1/n}\right\}\frac{\Vert \D_h^r \mathcal{L}_{n,\l}(f)\Vert_p^p}{h^{\a p}}\mathrm{d}\l\\
&\le C\( n^{\a p}\int_\T {\Vert \D_{1/n}^r \mathcal{L}_{n,\l}(f)\Vert_p^p}\mathrm{d}\l+ \sup_{h\ge 1/n} \frac{\w_r\(f,h\)_p^p}{h^{\a p}}\)\\
&\le C\sup_{h\ge 1/n} \frac{\w_r\(f,h\)_p^p}{h^{\a p}}.
  \end{split}
\end{equation*}
The last inequality implies (\ref{eqA3}).
\end{proof}

We will also use the following result, which can be obtained by repeating the proofs of formulas (6.1) and (6.3) from~\cite{RSFaaF}.
\begin{lemma}\label{lem3}
Let $f\in L_p$, $1\le p\le\infty$, $n\in \N$, and let $\{\mathcal{L}_{n,\l}\}$ be bounded operators in $L_p$. Then the following equivalence holds
\begin{equation*}
\Vert f-\mathcal{L}_n(f)\Vert_p\asymp \Vert f-\mathcal{L}_{n,\l}(f)\Vert_{\overline{p}}.
\end{equation*}
\end{lemma}

The lemma below gives an equivalent definition of (\ref{eqMeansH})--(\ref{eqMeansH2}).
\begin{lemma}\label{lem3XXX}
Let $f\in L_p$, $0<p\le \infty$, $0< \a \le r$, $r,n\in \N$, and let the family $\{\mathcal{L}_{n,\l}\}$ be bounded in $L_p$. Then
\begin{equation*}
   |f-\mathcal{L}_{n,\,\lambda}(f)|_{H_{\overline{p}}^{r,\a}}\asymp \(\int_\T |f-\mathcal{L}_{n,\,\lambda}(f)|_{H_{{p}}^{r,\a}}^p{\rm d}\l\)^\frac1p
\end{equation*}
with the usual modification in the case $p=\infty$.
\end{lemma}

\begin{proof}
The estimate from above is evident. Let us prove the estimate from below. As in the proof of Lemma~\ref{lem2}, for the convenience, we treat only the case $0<p<\infty$.

Let $T_n\in\mathcal{T}_n$, $n\in\N$, be such that $|f-T_n|_{H_{{p}}^{r,\a}}=\inf_{T\in\mathcal{T}_n}|f-T|_{H_{{p}}^{r,\a}}$. By using two times
Lemma~\ref{lem2} and (\ref{eqlem2pr300}), we obtain
\begin{equation*}
  \begin{split}
     \int_\T &|f-\mathcal{L}_{n,\,\lambda}(f)|_{H_{p}^{r,\a}}^p{\rm d}\l\\
     &\le C \(|f-T_n|_{H_{p}^{r,\a}}^p+\int_\T |T_n-\mathcal{L}_{n,\,\lambda}(T_n)|_{H_{p}^{r,\a}}^p{\rm d}\l+\int_\T |\mathcal{L}_{n,\,\lambda}(T_n-f)|_{H_{p}^{r,\a}}^p{\rm d}\l\)\\
     &\le C \(|f-T_n|_{H_{p}^{r,\a}}^p+ |T_n-\mathcal{L}_{n}(T_n)|_{H_{p}^{r,\a}}^p+\sup_{h\ge 1/n}\frac{\omega_r(f-T_n, h)_p^p}{h^{\a p}}\)\\
     &\le C \(|f-T_n|_{H_{p}^{r,\a}}^p+ \sup_{h>0}\int_\T \frac{\Vert \D_h^r (T_n-\mathcal{L}_{n,\,\lambda}(T_n))\Vert_p^p}{h^{\a p}}{\rm d}\l\)\\
     &\le C \(|f-T_n|_{H_{p}^{r,\a}}^p+|f-\mathcal{L}_{n,\,\lambda}(f)|_{H_{\overline{p}}^{r,\a}}^p+\sup_{h>0}\int_\T \frac{\Vert \D_h^r \mathcal{L}_{n,\,\lambda}(f-T_n)\Vert_p^p}{h^{\a p}}{\rm d}\l\)\\
      &\le C \(|f-T_n|_{H_{p}^{r,\a}}^p+|f-\mathcal{L}_{n,\,\lambda}(f)|_{H_{\overline{p}}^{r,\a}}^p\)\le  C |f-\mathcal{L}_{n,\,\lambda}(f)|_{H_{\overline{p}}^{r,\a}}^p,
   \end{split}
\end{equation*}
which proves the lemma.
\end{proof}

\section{Direct and inverse theorems. Properties of the best approximation in $H_p^{r,\a}$}

Let $0<p\le\infty$, $0<\a\le r$, $r,n\in\N$. Denote the best approximation in $H_p^{r,\a}$ by
$$
E_n(f)_{H_p^{r,\a}}=\inf_{T\in \mathcal{T}_n}\Vert f-T\Vert_{H_p^{r,\a}}
$$
and consider the following modulus of smoothness
$$
\theta_{k,\a}(f,\d)_p=\sup_{0<h\le \d}\frac{\w_k(f,h)_p}{h^\a},
$$
which was initially used for the investigation of approximation in the H\"older spaces (see, e.g.,~\cite{BJ}).

We obtain the following Jackson-type  theorem in terms of $\theta_{k,\a}(f,h)_p$.

\begin{theorem}\label{cor2}
   Let $f\in H_p^{r,\a}$, $0<p\le\infty$, $0<\a<\min(r,k)$ or $0<\a=k=r$, and $r,k,n\in\N$.
Then
\begin{equation}\label{eqR2}
  E_n(f)_{H_p^{r,\a}}\le C \theta_{k,\a}\(f,\frac1n\)_p,
\end{equation}
where $C$ is a constant independent of $n$ and $f$.
\end{theorem}

\begin{proof}
Let $\a<\min(r,k)$ and $T_n$, $n\in \N$, be polynomials of the best approximation of $f$ in $L_p$.
By Lemma~\ref{lemHolder} and Theorem~A it suffices to find an estimation for $|f-T_n|_{H_p^{k,\a}}$.

We have
\begin{equation}\label{eqcor2.1}
  |f-T_n|_{H_p^{k,\a}}\le \(\sup_{0<h<1/n}+\sup_{h\ge 1/n}\)\frac{\Vert \D_h^k (f-T_n)\Vert_p}{h^\a}=S_1+S_2.
\end{equation}
Using Theorem~A we obtain
\begin{equation}\label{eqcor2.2}
  S_2\le Cn^\a \Vert f-T_n\Vert_p\le C n^\a\w_k(f,1/n)_p\le C\theta_{k,\a}(f,1/n)_p.
\end{equation}
Moreover, for $S_1$ we estimate
\begin{equation}\label{eqcor2.3}
 \begin{split}
   S_1&\le C\(\sup_{0<h<1/n}\frac{\Vert \D_h^k f\Vert_p}{h^\a}+\sup_{0<h<1/n}\frac{\Vert \D_h^k T_n\Vert_p}{h^\a}\)\\
   &\le C\(\theta_{k,\a}(f,1/n)_p+\sup_{0<h<1/n}\frac{\Vert \D_h^k T_n\Vert_p}{h^\a}\).
 \end{split}
\end{equation}
To estimate the last term in (\ref{eqcor2.3}) we use Theorem~B and Theorem~A
\begin{equation}\label{eqcor2.4}
\begin{split}
  \sup_{0<h<1/n}\frac{\Vert \D_h^k T_n\Vert_p}{h^\a}&\le C n^\a \Vert \D_{1/n}^k T_n\Vert_p\le C n^\a(\Vert f-T_n\Vert_p+\Vert \D_{1/n}^k f\Vert_p)\\
&\le C n^\a \w_k(f,1/n)_p\le C \theta_{k,\a}(f,1/n)_p.
\end{split}
\end{equation}

Thus, combining (\ref{eqcor2.1})--(\ref{eqcor2.4}), we obtain (\ref{eqR2}) in the case $\a<\min(r,k)$. The same scheme one can use in the case  $r=\a=k$.
\end{proof}

The converse result is given by the following theorem.

\begin{theorem}\label{thR4}
  Let $f\in H_p^{r,\a}$, $0<p\le\infty$, $0<\a<\min(r,k)$  or $0<\a=k=r$, and $r,k,n\in\N$.
Then
\begin{equation}\label{eqR3}
  \theta_{k,\a}\(f,\frac 1n\)_p\le \frac{C}{n^{k-\a}}\(\sum_{\nu=0}^{n} (\nu+1)^{(k-\a) {p_1}-1} E_\nu(f)_{H_p^{r,\a}}^{p_1}\)^\frac1{p_1},
\end{equation}
where $C$ is a constant independent of $n$ and $f$.
\end{theorem}

\begin{proof}
As above we consider only the case $\a<\min(r,k)$.
Let $T_n\in \mathcal{T}_{n}$, $n\in \N$, be  polynomials of the best approximation of $f$ in $H_{p}^{r,\a}$.
Choosing $m\in \N\cup\{0\}$ such that $2^m\le n<2^{m+1}$, we get
\begin{equation}\label{eq.thR4.1}
  \t_{k,\a}(f,1/n)_{p}^{p_1}\le \t_{k,\a}(f-T_{2^{m+1}},1/n)_{p}^{p_1}+\t_{k,\a}(T_{2^{m+1}},1/n)_{p}^{p_1}.
\end{equation}
By the definition of the H\"older spaces and Lemma~\ref{lemHolder}, we obtain
\begin{equation}\label{eq.thR4.2}
 \t_{k,\a}(f-T_{2^{m+1}},1/n)_{p}\le |f-T_{2^{m+1}}|_{H_{p}^{k,\a}}\le E_{2^{m+1}}(f)_{H_{p}^{k,\a}}\le CE_{2^{m+1}}(f)_{H_{p}^{r,\a}}.
\end{equation}
By (\ref{eqA1}) in Theorem~B, we get
\begin{equation}\label{eq.thR4.3}
  \begin{split}
    \t_{k,\a}(T_{2^{m+1}},1/n)_{p}\le C n^{-(k-\a)}\Vert  T_{2^{m+1}}^{(k)}\Vert_{p}.
  \end{split}
\end{equation}
Using again inequality (\ref{eqA1}) and  Lemma~\ref{lemHolder}, we gain
\begin{equation}\label{eq.thR4.3000}
  \begin{split}
    \Vert  T_{2^{m+1}}^{(k)}\Vert_{p}^{p_1}
    &\le \Vert  (T_1- T_0)^{(k)}\Vert_{p}^{p_1}+\sum_{\mu=0}^m \Vert (T_{2^{\mu+1}}- T_{2^\mu})^{(k)}\Vert_{p}^{p_1}\\
    &\le C\(\w_k(T_1-T_0,1)_{p}^{p_1}+\sum_{\mu=0}^m 2^{\mu k p_1}\w_k\(T_{2^{\mu+1}}-T_{2^\mu},2^{-(\mu+1)}\)_{p}^{p_1}\)\\
    &\le C\( \Vert T_1-T_0\Vert_{H_{p}^{k,\a}}^{p_1} +\sum_{\mu=0}^m 2^{\mu (k-\a)p_1}  \Vert T_{2^{\mu+1}}-T_{2^\mu}  \Vert_{H_{p}^{k,\a}}^{p_1}  \)  \\
    &\le C\(E_0(f)_{H_{p}^{k,\a}}^{p_1}+\sum_{\mu=0}^m 2^{\mu(k-\a) p_1}E_{2^\mu}(f)_{H_{p}^{k,\a}}^{p_1} \)\\
    &\le C\(\sum_{\nu=0}^{n} (\nu+1)^{(k-\a) {p_1}-1} E_\nu(f)_{H_p^{r,\a}}^{p_1}\).
  \end{split}
\end{equation}
Thus, combining (\ref{eq.thR4.1})--(\ref{eq.thR4.3000}),
we get (\ref{eqR3}).

\end{proof}


Now let us consider the problem on the precise order of decrease of the best approximation in $H_p^{r,\a}$.

\begin{theorem}\label{th3KP}
                     Let $f\in H_p^{r,\a}$, $0<p\le \infty$, $0<\a< r$, $\a\le s$, and $r,s\in\N$.
There exists a constant $L>0$ such that for any $n\in\mathbb{N}$
                     \begin{equation}\label{rathore1P}
                           \t_{s,\a}\bigg(f,\frac 1n\bigg)_{p}\le LE_{n}(f)_{H_p^{r,\a}}
                     \end{equation}
                    iff for  some $k>s+1/{p_1}-1$ there exists a constant
                      $M>0$ such that for any $h\in(0,1]$
                     \begin{equation}\label{rathore2P}
                            \t_{s,\a}(f,h)_{p}\le M \t_{k,\a}(f,h)_{p}.
                     \end{equation}
\end{theorem}

\begin{remark}
  Note that the constants $M$ and $L$ in Theorem~\ref{th3KP} may depend on $f$.
\end{remark}

\begin{proof}
To prove this theorem one can use the scheme of the proof of the corresponding result from
\cite{R} in the case $1\le p\le \infty$ and from~\cite{K1} for $0<p<1$.
To do this one only needs to apply Theorem~\ref{corth3}, Theorem~\ref{thR4}, and the following two properties
of $\t_{k,\a}(f,\d)_p$, which can be easily obtained from (\ref{svmod1}) and (\ref{svmod2}),
\begin{equation}\label{svmod1T}
  \t_{k,\a}(f,\d)_p\le 2^\frac{k-r}{p_1}\t_{r,\a}(f,\d)_p,
\end{equation}
\begin{equation}\label{svmod2T}
  \t_{k,\a}(f,\l \d)_p\le k^{\frac1{p_1}-1}(\l+1)^{k-\a+\frac1{p_1}-1}\t_{k,\a}(f,\d)_p.
\end{equation}

Let us  present the proof only for the case $0<p<1$.

Let condition~(\ref{rathore2P}) be satisfied. Then from  (\ref{svmod1T}) and (\ref{svmod2T}) we get
\begin{equation}\label{rathoreqiP}
      \t_{k,\a}(f,\l h)_{p}\le 2^{\frac{k-s}p} s^{\frac1{p}-1} M (\l+1)^{s-\a-1+\frac 1p}\t_{k,\a}(f,h)_{p}
\end{equation}
for all $\l>0$ and $h\in(0,1]$.

Let us prove that
\begin{equation}\label{rathInvP}
       \frac{1}{n^{(k-\a)p}}\sum_{\nu=0}^n(\nu+1)^{(k-\a)p-1}E_\nu(f)_{H_p^{r,\a}}^p\le
       C M^p \t_{k,\a}\bigg(f,\frac 1n\bigg)_{p}^p,
\end{equation}
where $C$ is some positive constant independent of $f$ and $n$. Indeed, by using Theorem~\ref{cor2} and inequality (\ref{rathoreqiP}),
we obtain
\begin{equation*}
\begin{split}
      &\frac{1}{n^{(k-\a)p}}\sum_{\nu=0}^n(\nu+1)^{(k-\a)p-1}E_\nu(f)_{H_p^{r,\a}}^p\le
      \frac{C}{n^{(k-\a)p}}\sum_{\nu=0}^n(\nu+1)^{(k-\a)p-1} \t_{k,\a}\left(f,\frac
      {1}{\nu+1}\right)_{p}^p\\
      &\le\frac{CM^p}{n^{(k-s)p+p-1}} \t_{k,\a}\left(f,\frac 1n
      \right)^p_{p}\sum_{\nu=0}^n(\nu+1)^{(k-s)p+p-2}\le
      CM^p \t_{k,\a}\left(f,\frac 1n\right)_{p}^p.
\end{split}
\end{equation*}
Next, by using Theorem~\ref{thR4} and (\ref{rathInvP}), we get that for all $m, n\in\N$
\begin{equation}\label{eqRathorXXXP}
\begin{split}
       &\t_{k,\a}\(f,\frac{1}{mn} \)_{p}^p\le
       \frac{C}{(mn)^{(k-\a)p}}\sum_{\nu=0}^{mn}(\nu+1)^{(k-\a)p-1}E_{\nu}(f)_{H_p^{r,\a}}^p\\
       &=\frac{C}{(mn)^{(k-\a)p}}\(\sum_{\nu=n+1}^{mn}(\nu+1)^{(k-\a)p-1}E_{\nu}(f)_{H_p^{r,\a}}^p+
      \sum_{\nu=0}^n(\nu+1)^{(k-\a)p-1}E_{\nu}(f)_{H_p^{r,\a}}^p \)\\
       &\le  C\(\frac{1}{(mn)^{(k-\a)p}}\sum_{\nu=n+1}^{mn}
      (\nu+1)^{(k-\a)p-1}E_{\nu}(f)_{H_p^{r,\a}}^p+\frac{M^p}{m^{(k-\a)p}}
      \t_{k,\a}\left(f,\frac 1n\right)_{p}^p\).
\end{split}
\end{equation}
Inequality~(\ref{eqRathorXXXP}) implies that
\begin{equation*}
      \sum_{\nu=n+1}^{mn}(\nu+1)^{(k-\a)p-1}E_{\nu}(f)_{H_p^{r,\a}}^p\ge
      \frac{(mn)^{(k-\a)p}}{C}\t_{k,\a}\left(f,\frac{1}{mn}\right)_{p}^p-M^p n^{(k-\a)p}\t_{k,\a}\left(f,\frac
     1n\right)_{p}^p
\end{equation*}
and, by using the monotonicity of $E_n(f)_{H_p^{r,\a}}$ and
(\ref{rathoreqiP}), we derive
\begin{equation*}
      E_n(f)_{H_p^{r,\a}}^p\sum_{\nu=n+1}^{mn}(\nu+1)^{(k-\a)p-1}\ge (C m^{(k-s)p+p-1}-M^p)n^{(k-\a)p}
      \t_{k,\a}\left(f,\frac{1}{n} \right)_{p}^p.
\end{equation*}
Thus,  choosing $m$ appropriately we can find a positive constant  $C=C_{k,s,p,\a,M}$ such that
\begin{equation*}
      E_n(f)_{H_p^{r,\a}}^p\ge C \t_{k,\a}\left(f,\frac 1n\right)_{p}^p.
\end{equation*}
From the last inequality and (\ref{rathore2P}) we obtain (\ref{rathore1P}).

The reverse direction is an immediate consequence of Theorem~\ref{cor2}, which finishes the proof.
\end{proof}


Now let us establish connection between the errors of approximation in the spaces $H_p^{r,\a}$ and $L_p$.

\begin{lemma}\label{lem4}
 Let $f\in H_p^{r,\a}$, $0<p\le \infty$, $0<\a\le r$, and $r,n\in \N$. Then
\begin{equation}\label{eqlem4pr1}
  C_1 n^\a E_n(f)_p \le  E_n(f)_{H_p^{r,\a}}\le C_2\(n^\a E_n(f)_p+
\(\sum_{\nu=n}^\infty \nu^{\a {p_1}-1}E_\nu(f)_p^{p_1}\)^\frac1{p_1}\),
\end{equation}
where $C_1$ and $C_2$ are some positive constants independent of $n$ and $f$.
\end{lemma}

\begin{proof}
Let $T_n$, $n\in \N$, be such that
$
\Vert f-T_n\Vert_{H_p^{r,\a}}=E_n(f)_{H_p^{r,\a}}.
$
Then, by using Theorem~A, we obtain
\begin{equation*}
\begin{split}
n^\a E_n(f)_p &\le C n^\a\w_r(f-T_n,1/n)_p \\
&\le C\sup_{0<h\le 1/n}\frac{\Vert \D_h^r (f-T_n)\Vert_p}{h^{\a}}\le C E_n(f)_{H_p^{r,\a}}.
\end{split}
 \end{equation*}

Now let us prove the estimate from above. Let $T_n$, $n\in\N$,  be the polynomials of the best approximation of $f$ in $L_p$ and
let $m\in \N$ be such that $2^{m-1}\le n<2^m$.
We can write
\begin{equation}\label{eqPPRREEDD}
  f=T_{2^m}+\sum_{\nu=m}^\infty U_{2^\nu}\quad \text{in}\quad L_p,
\end{equation}
where $U_{2^\nu}=T_{2^{\nu+1}}-T_{2^\nu}$.
Using (\ref{eqPPRREEDD}) we have
\begin{equation}\label{eqlem4.1}
  |f-T_n|_{H_p^{r,\a}}^{p_1}\le |T_{2^m}-T_n|_{H_p^{r,\a}}^{p_1}+
\sum_{\nu=m}^\infty |T_{2^{\nu+1}}-T_{2^\nu}|_{H_p^{r,\a}}^{p_1}=S_1+S_2.
\end{equation}
By using Theorem~B, we obtain
\begin{equation}\label{eqlem4.2}
\begin{split}
  S_1&\le \(\sup_{0<h<2^{-m}}+\sup_{h\ge 2^{-m}}\)\frac{\Vert \D_h^r (T_{2^m}-T_n)\Vert_p^{p_1}}{h^{\a p_1}}\\
     &\le C 2^{\a p_1 m}\(\Vert \D_{2^{-m}}^r(T_{2^m}-T_n)\Vert_p^{p_1}+\Vert T_{2^m}-T_n\Vert_p^{p_1}\)\\
     &\le C 2^{\a p_1 m}\Vert T_{2^m}-T_n\Vert_p^{p_1}\le C n^{\a {p_1}}E_n(f)_p^{p_1}.
\end{split}
\end{equation}
Again, by using Theorem~B, we have
\begin{equation}\label{eqcor1pr1}
\begin{split}
S_2&\le \sum_{\nu=m}^\infty \sup_{0<h\le 2^{-\nu-1}}h^{-\a {p_1}}\Vert \D_h^r U_{2^\nu}\Vert_p^{p_1}+\sum_{\nu=m}^\infty \sup_{h\ge 2^{-\nu-1}}h^{-\a {p_1}}\Vert \D_h^r U_{2^\nu}\Vert_p^{p_1}\\
&\le C\sum_{\nu=m}^\infty 2^{\a {p_1} \nu}\Vert \D_{2^{-\nu-1}}^r U_{2^\nu}\Vert_p^{p_1}+\sum_{\nu=m}^\infty 2^{\a {p_1} \nu}\sup_{h\ge 2^{-\nu-1}}\Vert \D_h^r U_{2^\nu}\Vert_p^{p_1}\\
&\le C\sum_{\nu=m}^\infty 2^{\a {p_1} \nu}\Vert  U_{2^\nu}\Vert_p^{p_1}\le C\sum_{\nu=m}^\infty 2^{\a p_1 \nu}E_{2^\nu}(f)_p^{p_1}\le C\sum_{\mu=n}^\infty \mu^{\a {p_1}-1}E_\mu(f)_p^{p_1}.
\end{split}
\end{equation}
Thus, combining (\ref{eqlem4.1})--(\ref{eqcor1pr1}), we obtain the upper estimate in (\ref{eqlem4pr1}).
\end{proof}


One can see from the above lemma that $E_n(f)_{H_p^{r,\a}}$ can tend to zero very fast.
But at the same time if for some function $f\in L_p$ we have  $\theta_{r,\a}(f,\d)_p=o(\d^{r-\a})$, then $f\equiv \mathrm{const}$.
Besides, if $1\le p\le\infty$ and $f\not\equiv{\rm const}$, then $\t_{r,r}(f,\d)_p\ge C>0$ (see Lemma~\ref{lemHolder}).
Thus, estimates (\ref{eqR2}) and (\ref{eqR3}) are far from being sharp, because of the failure of $\theta_{r,\a}(f,\d)_p$. We introduce the new "modulus of smoothness", which, as we think,  will be more natural and useful in the Jackson-type theorem in the H\"older spaces $H_p^{r,\a}$. At least, the idea of this modulus of smoothness works for the strong converse inequalities in the H\"older spaces (see the next section).

Let $0<p\le\infty$, $0<\a\le r$, and $r,k\in\N$. Denote
$$
\psi_{k,r,\a}(f,\d)_{p}:= \sup\limits_{0<h\le \d}\frac{\w_k(\D_h^r f,\d)_p}{h^\a}.
$$

It is easy to see that for $f\in L_p$, $r,k\in \N$, $0<\a\le r$, and $\d>0$
\begin{equation}\label{eqPropNewM}
  \t_{k+r,\a}(f,\d)_p\le \psi_{k,r,\a}(f,\d)_p\le C \t_{\min(k,r),\a}(f,\d)_p,
\end{equation}
where $C$ is a constant independent of $f$ and $\d$.
%

Now let us consider the following improvement of Theorem~\ref{cor2}. Actually, in view of
(\ref{eqPropNewM}), it is an improvement only in the case $\a=r$.


\begin{theorem}\label{th3}
  Let $f\in H_p^{r,\a}$, $0<p\le\infty$, $0<\a\le r$, and $k,r,n\in\N$. Then
\begin{equation}\label{eqM77}
\begin{split}
  E_n(f)_{H_p^{r,\a}}&\le
                           C\begin{cases}
                             \displaystyle \psi_{k,r,\a}\(f,1/n\)_{p}, & \hbox{$1\le p\le\infty$,} \\
                             \displaystyle \(\int_0^{1/n}\(\frac{\w_{r+k}(f,t)_p}{t^\a}\)^{p}\frac{\mathrm{d}t}{t}\)^\frac1{p}, & \hbox{$0<p<1$.}
                           \end{cases}
                           \end{split}
\end{equation}
In particular, for all $0<p\le\infty$ we have
\begin{equation}\label{eqM77D}
E_n(f)_{H_p^{r,\a}}\le C \(\int_0^{1/n}\(\frac{\w_{r+k}(f,t)_p}{t^\a}\)^{p_1}\frac{\mathrm{d}t}{t}\)^\frac1{p_1},
\end{equation}
where $C$ is a constant independent of $f$ and $n$.
\end{theorem}

\begin{proof}

Inequality  (\ref{eqM77}) in the case $0<p<1$ is a direct consequence of Lemma~\ref{lem4} and Theorem~A, that is we have
\begin{equation}\label{EQU}
  E_n(f)_{H_p^{r,\a}}^p\le C\(n^{\a p}E_n(f)_p^p+\sum_{\nu=n}^\infty \nu^{\a p-1}E_\nu(f)_p^p\)\le C\int_0^{1/n}\(\frac{\w_{r+k}(f,t)_p}{t^\a}\)^p\frac{\mathrm{d}t}{t}.
\end{equation}

Let us consider the case $1\le p\le\infty$.
Let $V_n$, $n\in \N$, be linear polynomial operators of the form (\ref{eqI1}) such that
\begin{equation}\label{eqVJ}
  \Vert f-V_n(f)\Vert_p\le C_{k+r}\w_{k+r}(f,1/n)_p.
\end{equation}
By~\cite[pp. 204-205]{DL} such operators always exists.

Thus, by using (\ref{eqVJ}), we have
\begin{equation}\label{eqPP}
\begin{split}
  E_n(f)_{H_p^{r,\a}}&\le C\(\Vert f-V_n(f)\Vert_p+|f-V_n(f)|_{H_p^{r,\a}}\)\\
&\le C\(\w_{r+k}(f,1/n)_p+|f-V_n(f)|_{H_p^{r,\a}}\).
\end{split}
\end{equation}
It is evident that one only needs to estimate the second term of the right-hand side in (\ref{eqPP}).
We have
\begin{equation}\label{eqPP1}
\begin{split}
|f-V_n(f)|_{H_p^{r,\a}}\le \(\sup_{0<h<1/n}+\sup_{h\ge1/n}\)\frac{\Vert \D_h^r (f-V_n(f))\Vert_p}{h^\a}=S_1+S_2.
\end{split}
\end{equation}
Again by using (\ref{eqVJ}), we get
\begin{equation}\label{eqPP2}
\begin{split}
S_2&\le C n^\a \Vert f-V_n(f)\Vert_p \le Cn^\a \w_{r+k}(f,1/n)_p=Cn^\a \sup_{0<\d\le 1/n}\Vert \D_\d^k\D_\d^r f\Vert_p\\
&\le C n^\a \sup_{0<h\le 1/n} \sup_{0<\d\le 1/n}\Vert \D_\d^k\D_h^r f\Vert_p\le C \psi_{k,r,\a}\(f,\frac1n\)_{p}.
\end{split}
\end{equation}
To estimate $S_1$ we use the equality
$\D_h^r V_n(f)=V_n (\D_h^r f)$
and once again (\ref{eqVJ}). Thus, we have
\begin{equation}\label{eqPP3}
\begin{split}
S_1=\sup_{0<h\le 1/n}\frac{\Vert \D_h^r f-V_n(\D_h^r f)\Vert_p}{h^\a}\le C\psi_{k,r,\a}\(f,\frac1n\)_{p}.
\end{split}
\end{equation}
Hence, combining (\ref{eqPP1})--(\ref{eqPP3}), we get (\ref{eqM77}) for $1\le p\le\infty$, and (\ref{eqM77}) is proved.

To prove (\ref{eqM77D}) one can use (\ref{eqasympOw}).
\end{proof}

It turns out that under some natural condition on a function $f$ the modulus of smoothness $\psi_{k,r,\a}(f,1/n)_p$ is equivalent to the corresponding integral in (\ref{eqM77D}).

\begin{lemma}\label{lemCOND}
Let $f\in L_p$, $0<p\le \infty$,  $0<\a\le r$, and $k,r\in\N$. Suppose that there  exists a positive constant $C$ independent of $f$ and $\d$ such that
\begin{equation}\label{eqM5_33}
   \(\int_0^{\d}\(\frac{\w_{r+k}(f,t)_p}{t^\a}\)^{p_1}\frac{\mathrm{d}t}{t}\)^\frac1{p_1}\le C \frac{\w_{r+k}\(f,\d\)_p}{\d^{\a }},\quad 0<\d<1.
\end{equation}
Then
\begin{equation}\label{eqasympOw}
  \(\int_0^{\d}\(\frac{\w_{r+k}(f,t)_p}{t^\a}\)^{p_1}\frac{\mathrm{d}t}{t}\)^\frac1{p_1}\asymp \psi_{k,r,\a}\(f,\d\)_p,\quad 0<\d<1.
\end{equation}
\end{lemma}

\begin{remark}
 1) Note that the estimate from below in (\ref{eqasympOw}), which is used in the proof of Theorem~\ref{th3}, holds without assumption (\ref{eqM5_33}).

 2) Note that condition (\ref{eqM5_33}) is equivalent to
$$
\int_0^{\d}\frac{\w_{r+k}(f,t)_p}{t^\a}\frac{\mathrm{d}t}{t}\le C \frac{\w_{r+k}\(f,\d\)_p}{\d^{\a}},\quad 0<\d<1,
$$
where $C$ is a constant independent of $f$ and $\d$ (see~\cite[Corollary 4.10]{Tih04}).
\end{remark}

\begin{proof}
The estimation from above can be easily obtained from (\ref{eqM5_33}) and from
the corresponding estimation for $\d^{- \a} \w_{r+k}\(f,\d\)_p$ in (\ref{eqPP2}).

Let us prove the estimation from below. Let $T_{2^\nu}$, $\nu\in \Z_+$, be polynomials of the best approximation of $f$ in $L_p$. Let $n\in \Z_+$
be such that $2^{-(n+1)}\le \d<2^{-n}$.
We have
\begin{equation}\label{eqOw1}
\begin{split}
\psi_{k,r,\a}(f,\d)_p^{p_1}\le \psi_{k,r,\a}(f,2^{-n})_p^{p_1}\le \psi_{k,r,\a}(T_{2^n},2^{-n})_p^{p_1}+\psi_{k,r,\a}(f-T_{2^n},2^{-n})_p^{p_1}.
\end{split}
\end{equation}

To estimate the first term in the last inequality we use Theorem~A and Theorem~B. We obtain
\begin{equation}\label{eqOw2}
\begin{split}
\lefteqn{\psi_{k,r,\a}(T_{2^n},2^{-n})_p\le C2^{\a n} \Vert \D_{2^{-n}}^{r+k} T_{2^n}\Vert_p\le C 2^{\a n} \(\Vert f-T_{2^n}\Vert_p+ \Vert \D_{2^{-n}}^{r+k} f\Vert_p\)}\\
&\le C 2^{\a n}\w_{r+k}(f,2^{-n})_p\le C \d^{-\a}\w_{r+k}(f,\d/2)_p\\
&\le  C\(\int_{\d/2}^{\d}\(\frac{\w_{r+k}(f,t)_p}{t^\a}\)^{{p_1}}\frac{\mathrm{d}t}{t}\)^\frac1{{p_1}}\le  C\(\int_0^{\d}\(\frac{\w_{r+k}(f,t)_p}{t^\a}\)^{{p_1}}\frac{\mathrm{d}t}{t}\)^\frac1{{p_1}}.
\end{split}
\end{equation}
To estimate the second term in (\ref{eqOw1}) we use (\ref{eqcor1pr1}), (\ref{EQU}), and (\ref{svmod2}). We conclude
\begin{equation}\label{eqOw3}
\begin{split}
\lefteqn{\psi_{k,r,\a}(f-T_{2^n},2^{-n})_p^{p_1}\le C|f-T_{2^n}|_{H_p^{r,\a}}^{p_1}
\le C\sum_{\nu=n}^\infty 2^{\a {p_1} \nu}E_{2^\nu}(f)_p^{{p_1}}}\\
&\le C\int_0^{2^{-n}}\(\frac{\w_{r+k}(f,t)_p}{t^\a}\)^{{p_1}}\frac{\mathrm{d}t}{t}\le C\int_0^{\d}\(\frac{\w_{r+k}(f,t)_p}{t^\a}\)^{{p_1}}\frac{\mathrm{d}t}{t}.
\end{split}
\end{equation}

Thus, combining (\ref{eqOw1})--(\ref{eqOw3}), we obtain the estimate from below in (\ref{eqasympOw}).
\end{proof}

From Theorem~\ref{cor2} and inequality~(\ref{eqPropNewM}), Theorem~\ref{th3} and Lemma~\ref{lemCOND} we deduce the following assertion.

\begin{corollary}\label{corth3}
  Let $f\in H_p^{r,\a}$, $0<p\le\infty$, $0<\a\le r$, and $k,r\in\N$. For $0<p<1$ and $\a=r$ suppose also
  that $f$ satisfies condition (\ref{eqM5_33}).  Then
\begin{equation}\label{eqcorM77}
  E_n(f)_{H_p^{r,\a}}\le  C \psi_{k,r,\a}\(f,1/n\)_{p},
\end{equation}
where $C$ is a constant independent of $f$ and $n$.
\end{corollary}


Corollary~\ref{corth3} implies that one can replace the integral
by the modulus $\psi_{k,r,\a}(f,1/n)_p$ in (\ref{eqM77}) when $\a<r$ and $0<p<1$.
We do not know whether it is possible  in the case $\a=r$.
But we are inclined to believe that the answer is negative. Indeed, it is well-known that for any $f\in L_p$, $1< p<\infty$,
and for any $T_n$, $n\in \N$, such that $T_n\to f$ and $T_n'\to g$ in $L_p$ we have
$$
\Vert f-T_n\Vert_{H_p^{1,1}}\to 0\quad\text{as}\quad n\to\infty
$$
(see Lemma~\ref{lemHolderD} and~\cite[Lemma~4.5.5]{N}). In the case $0<p<1$ the situation is totally different. In particular,
the following proposition holds:
\begin{proposition}
  Let $0<p<1$. There exist a function $f\in L_p$ and polynomials $T_n$, $n\in \N$, such that
  $T_n \to f$ and $T_n'\to g$ in $L_p$, but
$$
\Vert f-T_n\Vert_{H_p^{1,1}}\ge C_f>0\quad\text{for sufficiently large}\quad n.
$$
\end{proposition}

\begin{proof}
We will use an example from~\cite{diti07}. Let
$$
f(x)=\left\{
       \begin{array}{ll}
        \displaystyle x, & \hbox{$x\in [0,\pi)$,} \\
        \displaystyle 2\pi-x, & \hbox{$x\in [\pi,2\pi]$}
       \end{array}
     \right.
$$
and $f(x)=f(x+2\pi)$. Let also
\begin{equation*}
g_{n}(x)=\left\{%
         \begin{array}{ll}
           \displaystyle\frac k{n}, & \hbox{$\displaystyle\frac{k}n\le x <\frac{k+1}{n}-\frac1{n^{2}}$,} \\
\phantom{.}\\
           \displaystyle\frac kn +\left(x-\frac{k+1}{n}+\frac1{n^{2}}\right)n, & \hbox{$\displaystyle\frac{k+1}{n}-\frac1{n^{2}}\le
x<\frac{k+1}{n}$}\\
         \end{array}%
       \right.
\end{equation*}
for $k=0,1,\dots, n-1$, $g_n(x)=1-g_n(x-1)$ for $1<x\le 2$, and $\vp_n(x)=\pi g_{n}(x/{\pi})$ for $x\in [0,2\pi]$.

We will need the following inequalities:
\begin{equation}\label{eqKop}
    \omega_1(\varphi_{n},1/n)_p\le Cn^{-1}\Vert \varphi_{n}'\Vert_p\le
Cn^{-\frac1p}.
\end{equation}
One can find the first inequality in~\cite{Kop}, but the second one can be verified by simple calculation. It is also easy to see that
$$
\Vert f -\varphi_{n}\Vert_p=\mathcal{O}\left(1/{n}\right).
$$

Let $T_{n}$, $n\in\N$, be polynomials of the best approximation of
$\varphi_{n}$ in  $L_p$, $0<p<1$. By using (\ref{eqI12}) and (\ref{eqKop}), we have
\begin{equation}\label{eqvL1}
\begin{split}
   \Vert f -T_{n}\Vert_p &\le C(\Vert
f-\varphi_{n}\Vert_p+\Vert \varphi_{n}-T_{n}\Vert_p)\\
&\le C(n^{-1}+\omega_1(\varphi_{n},1/n)_p)\le C n^{-1}.
\end{split}
\end{equation}
Taking into account Theorem~B and (\ref{eqKop}), we get
\begin{equation}\label{eqvLpDer}
    \Vert T_{n}'\Vert_p\le C n \omega_1(\varphi_{n},1/n)_p\le C
n^{1-\frac1p}.
\end{equation}
Thus, by using (\ref{eqA1}) and (\ref{eqvLpDer}), we obtain
\begin{equation*}
  \begin{split}
&\Vert f-T_n\Vert_{H_p^{1,1}}^p\ge \sup_{0<h\le 1/n}\frac{\Vert \D_h^1 (f-T_n)\Vert_p^p}{h^p}\\
&\ge \sup_{0<h\le 1/n}\frac{\Vert \D_h^1 f\Vert_p^p}{h^p}-\sup_{0<h\le 1/n}\frac{\Vert \D_h^1 T_n\Vert_p^p}{h^p}\ge \pi^p-C\Vert T_n'\Vert_p^p\ge \pi^p-Cn^{p-1}.
   \end{split}
\end{equation*}
\end{proof}

The following two theorems can be obtained in the same manner as Theorem~\ref{thR4} and  Theorem~\ref{th3KP} above.

\begin{theorem}\label{thR4con}
Let $f\in H_p^{r,\a}$, $0<p\le \infty$, $0<\a\le r$, $r,k\in \N$, and $n\in\N$.  Then
\begin{equation*}
 \psi_{k,r,\a}\(f,\frac 1n\)_{p}\le \frac{C}{n^{k}}\(\sum_{\nu=0}^{n} (\nu+1)^{k p_1-1} E_\nu(f)_{H_p^{r,\a}}^{p_1}\)^\frac1{p_1},
\end{equation*}
where $C$ is a constant independent of $f$ and $n$.
\end{theorem}

\begin{theorem}\label{th3K}
                     Let $f\in H_p^{r,\a}$, $0<p\le \infty$,  $r,s\in\N$, $0<\a\le r$.  Suppose also that $f$ satisfies condition (\ref{eqM5_33}) in the case $0<p<1$ and $\a=r$.
There exists a constant $L>0$ such that for any $n\in\mathbb{N}$
                     \begin{equation}\label{rathore1}
                           \psi_{s,r,\a}\bigg(f,\frac 1n\bigg)_{p}\le LE_{n}(f)_{H_p^{r,\a}},
                     \end{equation}
                    iff for  some $k>s+1/{p_1}-1$ there exists a constant
                      $M>0$ such that for any $h\in(0,1]$
                     \begin{equation}\label{rathore2}
                            \psi_{s,r,\a}(f,h)_{p}\le M \psi_{k,r,\a}(f,h)_{p}.
                     \end{equation}
\end{theorem}

In Theorem~\ref{th3K} we found the lower estimate for $E_n(f)_{H_p^{r,\a}}$
under conditions (\ref{rathore2}) and (\ref{eqM5_33}).
It turns out that for $E_n(f)_{H_p^{r,\a}}$ with $0<p\le 1$ there exists a non-trivial estimate from below for all
$f\in H_p^{r,\a}$ in terms of the special differences.

%
%
%
%


\begin{proposition}\label{pr2}
  If $f\in H_p^{r,\a}$, $0<p\le 1$, $0< \a\le r$, and $r\in\N$, then for all $n\in \N$
$$
n^{1-\frac1p}\sup_{h>0}\frac{\Vert \widetilde{(\D_h^r f)}_n\Vert_p}{h^\a}\le C_pE_n(f)_{H_p^{r,\a}},
$$
where
$$
\widetilde{f}_n(\l)=\frac1{4n+1}\sum_{j=0}^{4n}f(t_j+\l),\quad t_j=t_{j,n}=\frac{2\pi j}{4n+1}.
$$
\end{proposition}

\begin{proof}
Let
\begin{equation*}\label{eqA4}
  E_n^0(f)_p=\inf\left\{\Vert f-T\Vert_p\,:\, T\in \mathcal{T}_n,\quad \int_\T T(t){\rm d}t=0\right\}.
\end{equation*}
It turns out that
$$
c_p n^{1-\frac1p}\Vert \widetilde{f}_n\Vert_p\le E_n^0(f)_p\le C_{p}\(n^{1-\frac1p}\Vert \widetilde{f}_n\Vert_p+E_{n/2}(f)_p\).
$$
The estimation from above can be found in~\cite{K}. To estimate $E_n^0(f)_p$ from below let us note that for any $T_n\in \mathcal{T}_n$, $\int_\T T(t){\rm d}t=0$,
we have
$$
\widetilde{f}_n(\l)=\frac1{4n+1}\sum_{j=0}^{4n}(f(t_j+\l)-T_n(t_j+\l)).
$$
Applying this equality, we obtain
$$
\Vert \widetilde{f}_n\Vert_p^p\le \frac{C}{(4n+1)^{p}}\sum_{j=0}^{4n}\int_\T|f(t_j+\l)-T_n(t_j+\l)|^p\mathrm{d}\l\le Cn^{1-p}\Vert f-T_n\Vert_p^p.
$$
Now, let $T_n\in \mathcal{T}_n$, $n\in \N$, be polynomials of the best approximation in $H_p^{r,\a}$.
Thus, from the above inequality we have
$$
n^{1-\frac1p}\sup_{h>0}\frac{\Vert \widetilde{(\D_h^r f)}_n\Vert_p}{h^\a}\le C\sup_{h>0}\frac{E_n^0 (\D_h^r f)_p}{h^\a}
\le C\sup_{h>0}\frac{\Vert \D_h^r f -\D_h^r T_n \Vert_p}{h^\a}\le CE_n(f)_{H_p^{r,\a}},
$$
which proves the proposition.
\end{proof}

\section{Two-sided estimates of approximation by linear polynomial methods in $H_p^{r,\a}$}

To formulate the main theorems  in this section we need some auxiliary notations. For that purpose
let us introduce the general modulus of smoothness.
\begin{definition}
We will say that $w=w(\cdot,\cdot)_p\in \Omega_p=\Omega(L_p,\R_+)$, $0<p\le \infty$, if

\noindent 1) for $f\in L_p$ and for any $\d>0$ we have
\begin{equation}\label{eqM1}
  w(f,\d)_p\le C\Vert f\Vert_p;
\end{equation}
2) for $f,g\in L_p$ and for any $\d>0$ we have
\begin{equation}\label{eqM2}
  w(f+g,\d)_p\le C(w(f,\d)_p+w(g,\d)_p),
  \end{equation}
where $C$ is a constant independent of $f$, $g$, and $\d$.
\end{definition}


As a function $w$ we can take, for example, the classical modulus of smoothness $\w_k(f,\d)_p$ of arbitrary order $k$, or a corresponding
$K$-functional or its realization (see~\cite{DHI}), but one can also use more artificial objects, which were introduced and studied in~\cite{KT}.

\begin{definition}
For a function $w\in \Omega_p$ we define a ''modulus of smoothness'' related to the
H\"older space $H_p^{r,\a}$ as follows
\begin{equation}\label{eqM4}
  w(f,\d)_{H_p^{r,\a}}=w(f,\d)_p+\sup_{h>0}\frac{w(\D_h^r f,\d)_p}{h^\a}.
\end{equation}
\end{definition}

Let us consider some examples. It turns out that if we take $w(f,\d)_p=\Vert f\Vert_p$, then (\ref{eqM4}) defines the norm in H\"older spaces, which was introduced in (\ref{eqI4}). However, if $w(f,\d)_p=\w_k(f,\d)_p$, then formula (\ref{eqM4}) provides the definition of the corresponding modulus of smoothness in H\"older spaces $H_p^{r,\a}$, see the right-hand side of formula (\ref{eqI10}), and if $w(f,1/n)=\Vert f-\mathcal{L}_{n,\l}(f)\Vert_{\overline{p}}$, then in (\ref{eqM4}) we will have formula (\ref{eqMeansH}).

Now we are ready to formulate our main result in this section.

\begin{theorem}\label{th1}
Let $0<p\le \infty$, $0< \a\le r$, and $r\in \N$.  Let $\{\mathcal{L}_{n,\l}\}$ be bounded in $L_p$, ${w}_\mathcal{L} \in \Omega_p$,
and let us assume that the following equivalence holds for any $f\in L_p$ and $n\in \N$:
\begin{equation}\label{eqM5}
  \Vert f-\mathcal{L}_{n,\l}(f)\Vert_{\overline{p}}\asymp w_\mathcal{L}\(f,\frac 1n\)_p.
\end{equation}
Then for any $f\in H_p^{r,\a}$ and $n\in \N$ we have
\begin{equation}\label{eqM6}
  \Vert f-\mathcal{L}_{n,\l}(f)\Vert_{H_{\overline{p}}^{r,\a}}\asymp w_\mathcal{L}\(f,\frac 1n\)_{H_p^{r,\a}}+ \left\{
                                                                                             \begin{array}{ll}
                                                                                               \displaystyle E_n(f)_{H_p^{r,\a}}, & \hbox{$0<p<1$,} \\
                                                                                               \displaystyle 0, & \hbox{$1\le p\le\infty$.}
                                                                                             \end{array}
                                                                                           \right.
\end{equation}
\end{theorem}

\begin{proof}
We start from the case $0<p<1$.

Let us first prove the lower bound for $\Vert f-\L_{\l,n}(f)\Vert_{H_{\overline{p}}^{r,\a}}$.
It is evident that
\begin{equation}\label{eqth1pr1}
   E_n(f)_{H_p^{r,\a}}\le \Vert f-\L_{\l,n}(f)\Vert_{H_{\overline{p}}^{r,\a}}.
\end{equation}
Thus, by (\ref{eqth1pr1}) and (\ref{eqM5}) we only need to prove that
\begin{equation}\label{eqth1pr2}
  \sup_{h>0}\frac{w_\L(\D_h^r f,1/n)_p}{h^\a}
\le \Vert f-\L_{\l,n}(f)\Vert_{H_{\overline{p}}^{r,\a}}.
\end{equation}
By properties (\ref{eqM1}) and (\ref{eqM2}) we obtain
\begin{equation}\label{eqth1pr3}
\begin{split}
  w_\L(\D_h^r f,1/n)_p^p&=\frac1{2\pi}\int_{\T} w_\L(\D_h^r f,1/n)_p^p \mathrm{d}\l \\
&\le C\left(\Vert \D_h^r (f-\L_{n,\l}(f))\Vert_{\overline{p}}^p+\int_{\T} w_\L(\D_h^r \mathcal{L}_{n,\l}(f),1/n)_p^p \mathrm{d}\l\right)
\end{split}
\end{equation}
and, by (\ref{eqM5}) we arrive at
\begin{equation}\label{eqth1pr4}
\begin{split}
\int_{\T} w_\mathcal{L}(\D_h^r \L_{n,\l}(f),1/n)_p^p \mathrm{d}\l \le C\int_\T \int_\T \Vert \D_h^r \L_{n,\l}(f)-\L_{n,\b}(\D_h^r \L_{n,\l}(f))\Vert_p^p \mathrm{d}\l \mathrm{d}\b.
\end{split}
\end{equation}
Note that for any $\l,\b\in \R$ we have
\begin{equation}\label{eqth1pr5}
\L_{n,\b} (\D_h^r \L_{n,\l} (f))=\D_h^r \L_{n,\b}(\L_{n,\l}(f))
\end{equation}
and
\begin{equation}\label{eqth1pr6}
\L_{n,\l}-\L_{n,\b}\circ \L_{n,\l}=(\L_{n,\l}-I)+(I-\L_{n,\b})+\L_{n,\b}\circ (I-\L_{n,\l}),
\end{equation}
where $I$ is the identity operator.

Thus, by (\ref{eqth1pr4}), (\ref{eqth1pr5}) and (\ref{eqth1pr6}), we obtain
\begin{equation}\label{eqth1pr7}
\begin{split}
&\int_{\T} w_\mathcal{L}(\D_h^r \L_{n,\l}(f),1/n)_p^p \mathrm{d}\l
\le C\bigg(\Vert \D_h^r (\L_{n,\l}(f)-f)\Vert_{\overline{p}}^p\\&+\Vert \D_h^r (f-\L_{n,\b}(f))\Vert_{\overline{p}}^p+
\int_\T \int_\T \Vert \D_h^r \L_{n,\b}(f- \L_{n,\l}(f))\Vert_p^p \mathrm{d}\l \mathrm{d}\b    \bigg).
\end{split}
\end{equation}
By using Lemma~\ref{lem2} and Lemma~\ref{lem3XXX}, we also conclude
\begin{equation}\label{eqth1pr8}
\begin{split}
  \sup_{h>0}\int_\T \int_\T &\frac{\Vert \D_h^r \L_{n,\b}(f- \L_{n,\l}(f))\Vert_p^p}{h^{\a p}} \mathrm{d}\l \mathrm{d}\b \le \int_\T \int_\T  |\L_{n,\b}(f- \L_{n,\l}(f))|_{H_p^{r,\a}}^p \mathrm{d}\b \mathrm{d}\l\\
  &\le C \int_\T \sup_{h\ge 1/n}\frac{\w_r(f- \L_{n,\l}(f),h)_p^p }{h^{\a p}} \mathrm{d}\l\\
  &\le C \int_\T |f- \L_{n,\l}(f)|_{H_p^{r,\a}}^p \mathrm{d}\l\le C |f- \L_{n,\l}(f)|_{H_{\overline{p}}^{r,\a}}^p.
\end{split}
\end{equation}
Thus, combining (\ref{eqth1pr7}) and (\ref{eqth1pr8}), we have
\begin{equation*}
\sup_{h>0}\int_\T \frac{w_\L(\D_h^r \L_{n,\l}(f),1/n)_p^p}{h^{\a p}} \mathrm{d}\l \le C \vert f-\L_{\l,n}(f)\vert_{H_{\overline{p}}^{r,\a}}^p.
\end{equation*}
The last inequality together with (\ref{eqth1pr3}) implies (\ref{eqth1pr2}).

Now, let us prove the upper bound. It is sufficient to prove that
\begin{equation}\label{eqth1pr9}
|f-\L_{n,\l}(f)|_{H_{\overline{p}}^{r,\a}}\le C\left(\sup_{h>0}\frac{w_\L(\D_h^r f, 1/n)_p}{h^\a}+
\inf_{T\in \mathcal{T}_n}|f-T|_{H_{p}^{r,\a}}   \right).
\end{equation}
Let $T_n \in \mathcal{T}_n$ be an arbitrary polynomial.
Then
\begin{equation}\label{eqth1pr10}
\begin{split}
\Vert \D_h^r(f-\L_{n,\l}(f))\Vert_{\overline{p}}&\le C\(\Vert \D_h^r (f-T_n)\Vert_p+\Vert \D_h^r (T_n-\L_{n,\l}(T_n))\Vert_{\overline{p}}
+\Vert \D_h^r \L_{n,\l} (f-T_n)\Vert_{\overline{p}}\).
\end{split}
\end{equation}
By Lemma~\ref{lem2}, we have
\begin{equation}\label{eqth1pr11}
|\L_{n,\l} (f-T_n)|_{H_{\overline{p}}^{r,\a}} \le C |f-T_n|_{H_{p}^{r,\a}}.
\end{equation}
Note also that
\begin{equation}\label{eqth1pr12}
\D_h^r \L_{n,\l}(T_n)=\L_{n,\l}(\D_h^r T_n).
\end{equation}
Thus, by using (\ref{eqth1pr12}) and (\ref{eqM5}), we have for any $h>0$
\begin{equation}\label{eqth1pr13}
\Vert \D_h^r (T_n-\L_{n,\l}(T_n))\Vert_{\overline{p}}=\Vert \D_h^r T_n-\L_{n,\l}(\D_h^r T_n))\Vert_{\overline{p}}\le C w_\L(\D_h^r T_n ,1/n)_p
\end{equation}
and, by using (\ref{eqM1}) and (\ref{eqM2}), we obtain
\begin{equation}\label{eqth1pr14}
w_\L(\D_h^r T_n ,1/n)_p\le C\(w_\L(\D_h^r f ,1/n)_p+\Vert \D_h^r (f-T_n)\Vert_p\).
\end{equation}

Combining (\ref{eqth1pr10}), (\ref{eqth1pr11}), (\ref{eqth1pr13}), and (\ref{eqth1pr14}), we get (\ref{eqth1pr9}).

To prove the theorem in the case $1\le p\le \infty$ one can take into account Lemma~\ref{lem3} and the equality
$$
\Vert \D_h^r (f-\L_n(f))\Vert_p=\Vert \D_h^r f-\L_n(\D_h^r f)\Vert_p,
$$
which yields the desired result.
\end{proof}

By using the same arguments as in the above proof of Theorem~\ref{th1} one can easily show the following theorem for $p\ge 1$.
\begin{theorem}\label{th2}
Let $1\le p\le \infty$,  $0< \a\le r$, and $r\in \N$.  Let $\{\L_{n}\}$ be bounded in $L_p$, $w_\L \in \Omega_p$,
and let us assume that the following equivalence holds for any $f\in L_p$ and $n\in \N$:
\begin{equation*}
  \Vert f-\L_{n}(f)\Vert_{p}\asymp w_\L\(f,\frac 1n\)_p.
\end{equation*}
Then for any $f\in H_p^{r,\a}$ and $n\in \N$ we have
\begin{equation*}
  \Vert f-\L_{n}(f)\Vert_{H_{p}^{r,\a}}\asymp w_\L\(f,\frac 1n\)_{H_p^{r,\a}}.
\end{equation*}
\end{theorem}


In the case $0<p<1$ the problem is more complicated, which is based on the difficulty of  sharp estimates for $E_n(f)_{H_p^{r,\a}}$. But some relationship can be presented.

By using Theorem~\ref{th1}, Theorem~\ref{th3}, and Proposition~\ref{pr2} we can prove the following result.

\begin{corollary}\label{cor1}
Let $0<p<1$, $0<\a\le r$, $r,k\in \N$, and $w_\L \in \Omega_p$. Suppose also that $\{\L_{n,\l}\}$ is bounded in $L_p$
and that the following equivalence holds for any $f\in L_p$ and $n\in \N$:
\begin{equation}\label{eqM5_1}
  \Vert f-\L_{n,\l}(f)\Vert_{\overline{p}}\asymp w_\L\(f,\frac 1n\)_p.
\end{equation}
Then for any $f\in H_p^{r,\a}$ and $n\in \N$ we have the following two-sided estimate:
\begin{equation}\label{eqM6_1}
w_\L\(f,\frac 1n\)_{H_p^{r,\a}}+n^{1-\frac1p}\sup_{h>0}\frac{\Vert \widetilde{(\D_h^r f)}_n\Vert_p}{h^\a}\le  C\Vert f-\L_{n,\l}(f)\Vert_{H_{\overline{p}}^{r,\a}},
\end{equation}
\begin{equation}\label{eqM6_2}
C\Vert f-\L_{n,\l}(f)\Vert_{H_{\overline{p}}^{r,\a}}\le  w_\L\(f,\frac 1n\)_{H_p^{r,\a}}+\(\int_0^{1/n}\(\frac{\w_{r+k}(f,t)_p}{t^\a}\)^{p}\frac{\mathrm{d}t}{t}\)^\frac1{p}.
\end{equation}
\end{corollary}

\begin{remark}
Note that the second term in the right-hand side of (\ref{eqM6_1})--(\ref{eqM6_2}) can be replaced by $\theta_{r,\a}(f,1/n)_p$ (see Theorem~\ref{cor2}) or by
$\psi_{k,r,\a}(f,1/n)_p$, $k\in\N$, if $r>\a$ or if condition (\ref{eqM5_33}) holds for the function $f$  (see Theorem~\ref{th3} and Corollary~\ref{corth3}).
\end{remark}

For some classes of functions and for the classical moduli of smoothness one can also obtain sharp two-sided estimates.
Indeed, by using Theorem~\ref{th1} and Corollary~\ref{corth3} and taking into account that $\psi_{k,r,\a}(f,h)_p\le \w_k\(f,h\)_{H_p^{r,\a}}$, we get the following result.

\begin{corollary}\label{cor1_1}
Let $0<p<1$, $0<\a\le r$, $r\in \N$, and $k\in \N$. Suppose also that $\{\L_{n,\l}\}$ is bounded in $L_p$
and that the following equivalence holds for any $f\in L_p$ and $n\in \N$:
\begin{equation*}
  \Vert f-\L_{n,\l}(f)\Vert_{\overline{p}}\asymp \w_k\(f,\frac 1n\)_p.
\end{equation*}
If $r>\a$ or $f\in L_p$ satisfy condition  (\ref{eqM5_33}), then
we have the following equivalence:
\begin{equation*}
\Vert f-\L_{n,\l}(f)\Vert_{H_{\overline{p}}^{r,\a}}\asymp  \w_k\(f,\frac 1n\)_{H_p^{r,\a}}.
\end{equation*}
\end{corollary}

\begin{remark}
  Note that by Lemma~\ref{lem3XXX} one can replace $\Vert f-\L_{n,\l}(f)\Vert_{H_{\overline{p}}^{r,\a}}$ by $\(\int_\T |f-\mathcal{L}_{n,\,\lambda}(f)|_{H_{{p}}^{r,\a}}^p{\rm d}\l\)^{1/p}$ in Theorem~\ref{th1}, Corollary~\ref{cor1}, and Corollary~\ref{cor1_1}.
\end{remark}

\section{Some corollaries and concluding remarks}

Let us now discuss two-sided inequalities like (\ref{eqM6}) and (\ref{eqM5_1}) for $\theta_{r,\a}(f,\d)_p$. It turns out that in terms of $\theta_{r,\a}(f,\d)_p$ such inequalities, in general, do not hold. However, we can prove the following results, which can be of interest for some particular functions $f$.

\begin{theorem}\label{thR5}
  Let $0<p\le \infty$, $0\le\a\le r$, and $r\in \N$.
Suppose also that $\{\L_{n,\l}\}$ is bounded in $L_p$
and that the following equivalence holds for any $f\in L_p$ and $n\in \N$:
\begin{equation*}
  \Vert f-\L_{n,\l}(f)\Vert_{\overline{p}}\asymp \w_r\(f,\frac 1n\)_p.
\end{equation*}
Then the following equivalence holds for any $f\in H_p^{r,\a}$ and $n\in \N$:
\begin{equation}\label{eqR5}
  n^\a\w_r\(f,\frac1n\)_p+\Vert f-\L_{n,\l}(f)\Vert_{H_{\overline{p}}^{r,\a}}\asymp \theta_{r,\a}\(f,\frac 1n\)_p.
\end{equation}
\end{theorem}

Before proving Theorem~\ref{thR5} let us note that the estimate from above (\ref{eqR5}) is a simple corollary from Theorem~\ref{th1} and Theorem~\ref{cor2}. Concerning the estimate from below it turns out that these estimates do not hold without the first term in the left-hand side of (\ref{eqR5})  (see Proposition~\ref{pr3} below).

\begin{proof}
As it was mentioned above it is sufficient only to prove the estimation from below.
Let $h\in (0,1/n)$ be fixed and $T_{n}\in \mathcal{T}_n$, $n\in\N$, be polynomials of the best approximation in $H_p^{r,\a}$.
By using Theorem~B, we obtain
\begin{equation*}
  \begin{split}
      h^{-\a}\Vert \D_h^r f\Vert_p&\le C h^{-\a} \(\Vert \D_h^r (f-T_{n})\Vert_p+\Vert \D_h^r T_{n}\Vert_p\)\\
     &\le C \(h^{-\a}\Vert \D_h^r (f-T_{n})\Vert_p+n^\a\Vert \D_{1/n}^r T_{n}\Vert_p\)\\
     &\le C \(h^{-\a}\Vert \D_h^r (f-T_{n})\Vert_p+n^\a\Vert \D_{1/n}^r(f-T_n)\Vert_p +n^\a\Vert \D_{1/n}^r f\Vert_p\)\\
     &\le C \(\Vert  f-T_n\Vert_{H_{{p}}^{r,\a}}+n^\a\omega_r(f,1/n)_p\)\\
     &\le C \(\Vert  f-\L_{n,\l}(f)\Vert_{H_{\overline{p}}^{r,\a}}+n^\a\omega_r(f,1/n)_p\).
  \end{split}
\end{equation*}
Theorem~\ref{thR5} is proved.
\end{proof}

Now we show that the first term in the left-hand side of (\ref{eqR5}) cannot be dropped.

\begin{proposition}\label{pr3}
   Let $0<p\le \infty$, $0<\a\le r$, and $r\in \N$.
Suppose that $\{\L_{n,\l}\}$ is bounded in $L_p$
and that the following inequality holds for any $f\in L_p$ and $n\in \N$:
\begin{equation*}
  \Vert f-\L_{n,\l}(f)\Vert_{\overline{p}}\le C\w_r\(f,\frac 1n\)_p,
\end{equation*}
where $C$ is some constant independent of $f$ and $n$.
Then, for any non-constant function $f$, $f^{(r-1)}\in AC$, $f^{(r)}\in H_{p^*}^{r,\a}$ with $p^*=\max(1,p)$, and for any sequence
$\{\e_n\}$ with $\e_n\to 0+$ there holds
\begin{equation*}
  \frac{\theta_{r,\a}(f,1/n)_p}{\e_n n^\a\w_r(f,1/n)_p+\Vert f-\L_{n,\l}(f)\Vert_{H_{\overline{p}}^{r,\a}}}\to \infty\quad\text{as}\quad n\to \infty.
\end{equation*}
\end{proposition}

\begin{proof}
By using the inequality $\theta_{r,\a}(f,1/n)_p\ge n^\a \w_r(f,1/n)_p$ it is sufficient only to prove that
\begin{equation}\label{eqpr3pr1}
  \frac{\Vert f-\L_{n,\l}(f)\Vert_{H_{\overline{p}}^{r,\a}}}{n^\a\w_r(f,1/n)_p}\to 0 \quad\text{as}\quad n\to \infty.
\end{equation}

Suppose that there exists a constant $A_f>0$ such that for any $n\in\N$
\begin{equation}\label{eqpr3pr2}
  n^\a   \w_r(f,1/n)_p   \le A_f    \Vert f-\L_{n,\l}(f)\Vert_{H_{\overline{p}}^{r,\a}}.
\end{equation}
Recall, that if $f^{(r-1)}\in AC$ and $f^{(r)}\in L_p$, $1\le p\le \infty$, then
\begin{equation}\label{eqpr3pr3}
  \w_r(f,\d)_p\le C_{r,p} \d^r \Vert f^{(r)}\Vert_p
\end{equation}
and
\begin{equation}\label{eqpr3pr3dop}
  \w_{2r}(f,\d)_p\le C\d^r\w_{r}(f^{(r)},\d)_p \le C \d^{r+\a}\Vert f^{(r)}\Vert_{H_{p}^{r,\a}}
\end{equation}
(see~\cite[p. 53]{DL}).
Thus, by using H\"older's inequality, and (\ref{eqpr3pr3}), we get that for any $h>0$
\begin{equation}\label{eqpr3pr4}
  \w_r(\D_h^r f,\d)_p\le C \w_r(\D_h^r f,\d)_{p^*}\le C\d^r \Vert \D_h^r f^{(r)}\Vert_{p^*}.
\end{equation}
By using~(\ref{eqM77D}) and (\ref{eqpr3pr3dop}) we also have for any $n\in\N$
\begin{equation}\label{eqpr3pr44}
  E_{n}(f)_{H_p^{r,\a}}\le Cn^{-r}\Vert f^{(r)}\Vert_{H_{p^*}^{r,\a}}.
\end{equation}

Thus, combining (\ref{eqpr3pr2}), (\ref{eqpr3pr4}), (\ref{eqpr3pr44}) with (\ref{eqM6}), we finally obtain
\begin{equation*}\label{eqpr3pr6}
  \w_r\(f,1/n\)_p\le CA_f n^{-\a}\(\w_r(f,1/n)_{H_p^{r,\a}}+E_n(f)_{H_p^{r,\a}}\)\le CA_f n^{-r-\a}\Vert f^{(r)}\Vert_{H_{p^*}^{r,\a}}.
\end{equation*}
However, if $f^{(r-1)}\in AC$ and $f\not\equiv{\rm const}$, then the last statement is impossible, see Lemma~1.5 in~\cite{SKO} and Theorem~3.5 in~\cite{diti07}.

Thus, we have proved (\ref{eqpr3pr1}) and hence our proposition.
\end{proof}

The following two corollaries can be obtained by a standard scheme (see, for example,~\cite{BJ}, \cite{PrPr}, and~\cite{Pr87}).
In particular, by Theorem~\ref{th3} we obtain the following result.

\begin{corollary}
  Let $0<p\le\infty$, $0< \b\le \a\le r$, $r\in \N$, and $n\in \N$. Then the following inequalities hold for any $f\in H_p^{r,\a}$:
  \begin{equation*}
    E_n(f)_{H_p^{r,\b}}\le \frac C{n^{\a-\b}}\Vert f\Vert_{H_p^{r,\a}}
  \end{equation*}
  and
    \begin{equation*}
    E_n(f)_{H_p^{r,\b}}\le \frac C{n^{\a-\b}}E_n(f)_{H_p^{r,\a}}.
  \end{equation*}
\end{corollary}

\begin{corollary}
  Let $0<p\le\infty$, $0< \b\le \a\le r$, $r,k\in \N$, $f\in H_p^{r,\a}$, and $0<\g<k$. Suppose also $f$ satisfies condition (\ref{eqM5_33}) in the case $0<p<1$ and $\a=r$. Then the following assertions are equivalent:

$(i)$
$
\psi_{k,r,\a}(f,h)_{p}=\mathcal{O}(h^\g),\quad h\to +0,
$

$(ii)$
$
E_n(f)_{H_p^{r,\a}}=\mathcal{O}\(n^{-\g}\),\quad n\to\infty,
$

$(iii)$
$
E_n(f)_{p}=\mathcal{O}\(n^{-\g-\a}\),\quad n\to\infty,
$

$(iv)$
$
E_n(f)_{H_p^{r,\b}}=\mathcal{O}\(n^{-\g-\a+\b}\),\quad n\to\infty.
$

If, in addition, $\a+\g<k$, then

$(v)$
$
\omega_k(f,h)_{p}=\mathcal{O}(h^{\g+\a}),\quad h\to +0.
$

\end{corollary}

\begin{proof}
The equivalence $(i) \Leftrightarrow (ii)$ follows form Theorem~\ref{thR4con} and Corollary~\ref{corth3}. The equivalences $(ii) \Leftrightarrow (iii) \Leftrightarrow (iv)$ follows form Lemma~\ref{lem4}. The equivalence $(iii) \Leftrightarrow (v)$ is standard,  see~\cite[Ch. 7, \S 3]{DL}, see also Theorem~A and~\cite[Theorem~5]{I} for the case $0<p<1$.
\end{proof}

\begin{remark}\label{rem1}
In all above results we suppose that $r\ge \a$. This assumption is essential. Indeed, if $\a>r$, then the quantity $E_n(f)_{H_p^{r,\a}}$ is not
well-defined. For example, let the function $f\not \equiv const$ be such that $\w_r(f,h)_p=\mathcal{O}(h^{r})$. Then, by using Theorem~B, we have for any $0<h\le 1/n$
 $$
 C_1h^{r-\a}\Vert T_n^{(r)} \Vert_p-C_1 \Vert f\Vert_{H_p^{r,\a}}\le E_n(f)_{H_p^{r,\a}}.
 $$
That is $E_n(f)_{H_p^{r,\a}}=\infty$.
\end{remark}

\begin{remark}
  Some of the above results remain true in more general H\"older spaces $H_p^{r,\w}$, with the norm
  $$
  \Vert f\Vert_{H_p^{r,\w}}=\Vert f\Vert_p+\sup_{h>0}\frac{\Vert \D^r_h f\Vert_p}{\w(h)},
  $$
  where the function $\w$ is some modulus of continuity such that the fucntion
  $
  \frac{\w(h)}{h^\a}
  $
  is monotonically decreasing, see e.g.~\cite{PrPr}.
\end{remark}

\medskip

\textbf{Acknowledgements} The authors are indebted to the referees for a
thorough reading and valuable suggestions, which allowed us to improve this paper considerably.

\end{document}